\documentclass[10pt,reqno]{amsart}

\usepackage{graphicx}
\usepackage{multirow}
\usepackage{amsmath,amssymb,amsfonts}
\usepackage{amsthm}
\usepackage{mathrsfs}
\usepackage[title]{appendix}
\usepackage{xcolor}
\usepackage{indentfirst}
\usepackage{titlesec}

\titleformat{\section}[block]
  {\normalfont\Large\bfseries}
  {\thesection}{1em}{}
\titlespacing*{\section}
  {0pt}{2\baselineskip}{1\baselineskip}

\titleformat{\subsection}[block]
  {\normalfont\bfseries}
  {\thesubsection}{1em}{}
\titlespacing*{\subsection}
  {0pt}{1.5\baselineskip}{0.5\baselineskip}

\titleformat{\subsubsection}[block]
  {\normalfont\normalsize\bfseries}
  {\thesubsubsection}{1em}{}

\titlespacing*{\subsubsection}
  {0pt}{1.25\baselineskip}{0.4\baselineskip}

\usepackage{textcomp}
\usepackage{manyfoot}
\usepackage{booktabs}
\usepackage{algorithm}
\usepackage{algorithmicx}
\usepackage{algpseudocode}
\usepackage{listings}
\usepackage{tikz}
\usepackage{tikz-cd}
\usepackage[utf8]{inputenc}

\usetikzlibrary{
  arrows.meta,
  positioning,
  calc,
  decorations.pathreplacing
}

\usepackage{mathtools}
\usepackage{enumitem}

\setlist[itemize]{
  leftmargin=*,
  labelindent=\parindent,
  before=\vspace{0.5\baselineskip},
  after=\vspace{0.5\baselineskip}
}

\usepackage{hyperref}

\hypersetup{
  colorlinks=true,
  linkcolor=blue,
  citecolor=blue,
  urlcolor=cyan,
  pdftitle={Distances on Constraint Representations},
  pdfauthor={Kohei Hatano}
}
\mathtoolsset{showonlyrefs}
\numberwithin{equation}{section}

\newtheorem{prop}{Proposition}[section]
\newtheorem{lem}[prop]{Lemma}
\newtheorem{ex}[prop]{Example}
\newtheorem{dfn}[prop]{Definition}%
\newtheorem{thm}[prop]{Theorem}%
\newtheorem{assume}[prop]{Assumption}%

\begin{document}
\title[Distances on Constraint Systems]{Distances on Finite Constraint Systems and LICQ Radii in Nonlinear Programming}
\author{Kohei Hatano}
\thanks{Information Technology R\&D Center, Mitsubishi Electric Corporation.}
\thanks{Email: \texttt{Hatano.Kohei@ce.MitsubishiElectric.co.jp}}
\begin{abstract}

Mathematical programming formulations are routinely modified through
algebraic reformulations, perturbations of constraint functions, and the
addition or removal of constraints. Some of these modifications preserve the
feasible set, whereas others reflect newly imposed modeling requirements and
therefore change the optimization problem itself. This paper develops a
metric framework for quantifying the robustness of formulation-dependent
properties. Formulations are modeled as typed collections of constraint
functions. Within each fixed typed-cardinality class, their distance is
defined by optimally matching constraints of the same type and measuring the
resulting discrepancies in a norm on the underlying function space.

As a concrete demonstration of the framework, we study the Linear Independence
Constraint Qualification (LICQ). We introduce the pointwise LICQ radius,
defined as the distance from a feasible formulation to the set of formulations
that remain feasible at a prescribed point but fail LICQ there. Under suitable
assumptions, we derive an explicit formula for this radius and characterize the
effect of adding a single equality or inequality constraint.

\end{abstract}


\maketitle
\section{Introduction}
Constraint formulations in mathematical programming are often modified during the operational phase. 
Specifically, such modifications may include the addition or deletion of constraints in response to changes in the requirements of the target problem, 
as well as changes in the coefficients and right-hand-side constants appearing in the constraint functions. 
In practice, an important issue is determining the extent to which existing software and algorithms must be modified in response to such changes.

On the other hand, it is generally difficult to determine, before actually solving the problem, 
how changes in the constraint formulation will affect the difficulty of solving the problem. 
For example, even when two formulations represent the same feasible set, formulation-dependent 
properties—such as the gradients of the constraint functions and the structure of the active constraints—may differ. 
Such differences can affect the numerical stability of computations and the behavior of algorithms.

Motivated by this issue, the ultimate goal of this study is to quantify properties that depend on the formulation 
of an optimization problem and to construct a mathematical framework for evaluating, 
in advance and without actually solving the problem, the potential effects of formulation changes on the properties of the optimization problem and the stability of algorithms.

As a first step toward this goal, this paper considers continuous optimization problems 
and introduces a distance structure on the space of constraint formulations consisting of equality and 
inequality constraints. Furthermore, focusing on LICQ as one of the regularity conditions that depends on the formulation, 
we define the LICQ radius as the minimum perturbation distance required for LICQ to fail at a given point. We interpret the LICQ radius 
as an indicator of the robustness of LICQ with respect to perturbations of the formulation and derive an explicit formula for it.

Constraint qualifications are regularity conditions imposed at feasible points. 
They play an important role in optimality conditions, multiplier theory, sensitivity analysis, 
and numerical algorithms \cite{BonnansShapiro2000,NocedalWright2006,RockafellarWets1998}. 
We focus on LICQ as a representative constraint qualification. At a feasible point $x$, 
LICQ requires the gradients of all equality constraints and all inequality constraints active at $x$ 
to be linearly independent. In particular, whenever KKT multipliers exist, LICQ implies their uniqueness. 
Because a perturbation of the formulation can change the set of active constraints—for instance, 
by causing a previously inactive inequality to become active—a quantitative robustness measure 
for LICQ must account for both the perturbations of active gradients and the possible activation of inactive constraints.

A simple example illustrates why LICQ is a property of the formulation and
not of the feasible set alone. Consider the open ball
\begin{equation}
U_1=\{x\in\mathbb R^2\mid x_1^2+x_2^2<1\}
\end{equation}
and the two inequality systems
\begin{align}
    P_1&:\quad x_1\leq 0,\quad x_2\leq 0,\\
    P_2&:\quad x_1\leq 0,\quad x_2\leq 0,\quad x_1+x_2\leq 0.
\end{align}
They define the same feasible set,
\begin{equation}
\{x\in U_1\mid x_1\leq 0,\ x_2\leq 0\}.
\end{equation}
At the origin, the active gradients of $P_1$ are
\begin{equation}
(1,0),\qquad (0,1),
\end{equation}
whereas those of $P_2$ are
\begin{equation}
(1,0),\qquad (0,1),\qquad (1,1).
\end{equation}
Thus LICQ holds for $P_1$ at the origin but fails for $P_2$, even though the
two systems have the same feasible set. Notice that these two formulations
belong to different typed-cardinality strata, since they have different
numbers of inequality constraints. The example therefore illustrates a
difference between formulations, rather than a distance between them in the
metric constructed in this paper.

We regard a finite system of equality and inequality constraints as a typed
multiset of functions. For each fixed number of equality and inequality
constraints, we equip the corresponding formulation space with a matching
distance induced by the norm of the underlying function space. This distance
is invariant under permutations within each constraint type, while preserving
the distinction and multiplicity of the constraints.

The metric is defined only within each fixed typed-cardinality stratum. We do
not introduce distances between formulations with different numbers of
equality or inequality constraints. Constraint additions and removals are
instead represented by maps between strata and are analyzed through their
effects on the pointwise LICQ radius.

For a feasible formulation and a reference point, the pointwise LICQ radius
measures the smallest same-stratum perturbation that preserves feasibility at
that point and causes LICQ to fail. It is therefore a local,
formulation-dependent measure of LICQ robustness. The radius depends on the
chosen function space, norm, and scaling of the constraint functions, and is
not intended to be an invariant of the feasible set.

At a fixed point, feasibility and LICQ depend only on the values and first
derivatives of the constraint functions. Under an appropriate first-order
evaluation assumption, the radius can therefore be studied in a
finite-dimensional first-order jet space. This leads to an explicit formula
that accounts for both the activation of inequality constraints and the
subsequent loss of linear independence among the active gradients.

The remainder of the paper develops this framework and analyzes the
pointwise LICQ radius. The main contributions are summarized below.

\subsection*{Contributions}

The main contributions of this paper are as follows.

\begin{itemize}

\item For every fixed pair $(p,q)$, we introduce the typed formulation space
$\mathcal C_{p,q}$ of systems with $p$ equality constraints and $q$
inequality constraints, modeled as a typed symmetric product. We equip this
space with a matching distance that is invariant under permutations of
constraints within each type while retaining constraint multiplicities and
constraint types
(Definition~\ref{def:constraint-matching-distance}).

\item For a formulation $P\in\mathcal C_{p,q}$ and a feasible point $x$, we
define the pointwise LICQ radius as the distance, within the same stratum
$\mathcal C_{p,q}$, to pointwise LICQ failure
(Definition~\ref{def:LICQrad}).

\item Under the first-order evaluation assumptions introduced in
Section~\ref{sec:radius}, we prove that the pointwise LICQ radius on
$\mathcal C_{p,q}$ agrees with the corresponding radius on a finite-dimensional
first-order jet formulation space
(Proposition~\ref{prop:equivalence-radii}).

\item We derive an explicit formula for the pointwise LICQ radius as a finite
minimum over subsets of inequality constraints
(Theorem~\ref{thm:explicit-radius}). Each term in the formula consists of the
cost of making the selected inequalities active at the reference point and
the cost of making the resulting active-gradient matrix rank deficient.

\item We analyze the effect of adding one equality or one inequality
constraint by means of augmentation maps between typed-cardinality strata.
The resulting radius formulas
(Theorems~\ref{thm:augmentation-ineq} and
\ref{thm:augmentation-eq}) follow as direct consequences of the explicit
radius formula.

\end{itemize}
\newpage
\subsection*{Organization of the paper}

The paper is organized as follows.
Section~\ref{sec:related} discusses related work on perturbation analysis,
constraint qualifications, distances to ill-posedness, and metric
regularity. Section~\ref{sec:preliminaries} fixes the notation, introduces
the function-space setting, and defines formulation spaces as typed symmetric
products. Section~\ref{sec:distance} defines matching distances on fixed
typed-cardinality strata and establishes their basic properties, including
stability properties of augmentation maps. Section~\ref{sec:radius}
introduces the pointwise LICQ radius and the first-order jet formulation
space. Sections~\ref{sec:Explicit} and~\ref{sec:augmentation} contain the
main results. Section~\ref{sec:Explicit} derives an explicit formula for the
pointwise LICQ radius under additional first-order assumptions, while
Section~\ref{sec:augmentation} analyzes the effect of adding a single
equality or inequality constraint.

\section{Related Works}\label{sec:related}
\subsection{Perturbation analysis and variational convergence in optimization}
\label{subsec:perturbation-analysis}

Perturbation and sensitivity analysis studies how the behavior of an
optimization problem depends on changes in its data. In nonlinear
programming, central objects include local minimizers, optimal values,
Lagrange multipliers, and KKT systems. Classical treatments include
Fiacco~\cite{Fiacco1983} and Bonnans and Shapiro~\cite{BonnansShapiro2000}.
Variational analysis places these questions in a broader setting through
notions such as set convergence, epi-convergence, and graphical convergence;
see Attouch~\cite{Attouch1984}, Attouch and Wets~\cite{AttouchWets1991},
and Rockafellar and Wets~\cite{RockafellarWets1998}.

Quantitative approaches in this direction measure discrepancies between
problem data and relate them to changes in solutions or other objects of
interest. For example, Royset~\cite{Royset2020Stability} obtains stability
and error estimates for optimization problems and generalized equations
using truncated Hausdorff distances between graphs and epigraphs.
Chieu, Trang, and Yen~\cite{ChieuTrangYen2025} study tilt stability of
local minimizers of nonlinear programs under the relaxed constant rank
constraint qualification. These works illustrate distinct stability
questions, whose answers depend on both the property under consideration
and the permitted perturbations.

The present paper studies the robustness of LICQ for a finite constraint
formulation at a prescribed feasible point. The formulation records the
individual constraint functions, their equality or inequality types, and
their multiplicities. Consequently, two formulations defining the same
feasible set need not have the same LICQ behavior. Our metric compares
these function data while identifying permutations within each constraint
type. The property whose robustness is measured is therefore specified
at the level of the constraint formulation: we seek the smallest
perturbation that preserves feasibility at the chosen point and destroys
linear independence of the active gradients. No objective function or
optimal solution is needed to define this quantity.

The use of a metric on problem data is consistent with the preceding
stability literature. Likewise, the matching construction uses standard
quotient-metric and optimal-assignment ideas. For a fixed positive
cardinality $r$, the distance on $\operatorname{Sym}^r(W)$ in
Section~\ref{sec:distance} equals $\sqrt r$ times the $2$-Wasserstein
distance between the corresponding uniform empirical measures, with
ground distance induced by $\|\cdot\|_W$; see
Peyr\'e and Cuturi~\cite{PeyreCuturi2019} for the discrete assignment
formulation. Here this construction is applied separately to equality
and inequality constraints. Its role is to provide a permutation-invariant
setting for quantifying LICQ failure while retaining the data on which
LICQ depends.

\subsection{Distances to ill-posedness and metric regularity}
\label{subsec:ill-posedness-metric-regularity}

Measuring robustness by the distance to failure is an established approach
to conditioning. Demmel~\cite{Demmel1987} studies the relationship between
condition numbers and distances to ill-posedness in numerical analysis.
Renegar~\cite{Renegar1995,Renegar2001} develops distance-based condition
theory for linear programming, and Freund and Vera~\cite{FreundVera1999}
study condition measures and distances to ill-posedness for conic linear
systems. The LICQ radius considered here follows this general principle,
with the failure set specified by feasibility and loss of LICQ at a
fixed point.

A particularly relevant comparison is the radius of metric regularity.
Dontchev, Lewis, and Rockafellar~\cite{DontchevLewisRockafellar2003}
study the smallest perturbation that destroys metric regularity at a
specified graph point and establish radius formulas for linear and
locally Lipschitz perturbations. Gfrerer and
Kruger~\cite{GfrererKruger2023} extend the exact formula for Lipschitz
perturbations to general Asplund spaces. The related radius of metric
subregularity is studied by Dontchev, Gfrerer, Kruger, and
Outrata~\cite{DontchevGfrererKrugerOutrata2020}.
These are already local radius theories. The distinction in the present
paper lies in the failure property and the perturbation model.

First, LICQ concerns linear independence of all equality gradients and
active inequality gradients. Its failure need not entail the failure of
weaker constraint qualifications. For example, the two inequalities
$t\le0$ and $t\le0$ fail LICQ at zero, whereas the direction $d=-1$
strictly decreases both constraints and hence verifies the
Mangasarian--Fromovitz constraint qualification. This separates the
failure property studied here from the regularity property of the
usual inequality constraint mapping.

Second, we measure perturbations by the $W$-norms of the constraint
functions and require the prescribed point to remain feasible.
In the locally Lipschitz radius formula of Dontchev, Lewis, and
Rockafellar~\cite{DontchevLewisRockafellar2003}, perturbation size is
measured by a local Lipschitz modulus, and the reference output is
shifted by the perturbation value. In our setting, the induced quotient
norm on first-order data accounts for changes in both function values
and gradients. Making a previously inactive inequality active therefore
has a cost determined by the same norm that measures the subsequent
gradient perturbation.

The specific results of this paper quantify that cost exactly.
Proposition~\ref{prop:equivalence-radii} reduces the LICQ radius in the
function space to the radius in a first-order jet space equipped with
its quotient norm. Under the Hilbert-space and surjectivity assumptions,
Theorem~\ref{thm:explicit-radius} then expresses the radius as a finite
minimum over subsets of inequalities. Each candidate combines an
activation cost with a weighted distance to row-rank deficiency; the
gradient matrix is corrected to account for coupling between values
and gradients in the quotient norm. The singular-value step uses
classical matrix approximation. The contribution is the resulting
exact distance formula for LICQ failure under function perturbations,
including changes in the active set, on the typed formulation space.
The constraint-addition formulas in
Theorems~\ref{thm:augmentation-ineq} and~\ref{thm:augmentation-eq}
are consequences of this characterization.

\section{Mathematical preliminaries and constraint formulations}
\label{sec:preliminaries}

This section introduces the basic framework used throughout the paper. We first
fix an ambient function space for constraint functions. We then define constraint
formulations as typed symmetric products of this function space. This construction
treats equality and inequality constraints as different types of functions, ignores the
ordering within each type, and retains multiplicities.

\subsection{Basic notation and definitions}

\begin{itemize}
\item For any positive integer $r$, $S_r$ denotes the symmetric group of degree $r$.

\item For any set $X$, we consider the natural action of $S_r$ on the direct
product $X^r$ by permutations. The quotient space is denoted by
\begin{equation}
    \operatorname{Sym}^r(X)=X^r/S_r.
\end{equation}
For any element $(x_1,\ldots,x_r)\in X^r$, its equivalence class is denoted by
\begin{equation}
    \overline{(x_1,\ldots,x_r)}.
\end{equation}

\item For a normed space $W$, $\|\cdot\|_W$ denotes the norm on $W$.

\item For a matrix $A$, $\sigma_{\min}(A)$ denotes the smallest singular value
of $A$, and $\|A\|_F$ denotes the Frobenius norm. The notation $A\succ0$
means that $A$ is symmetric positive definite, and $A\succeq0$ means that
$A$ is symmetric positive semidefinite.

\item For a constraint formulation $P$, $K(P)$ denotes the feasible set of $P$.
\end{itemize}

\subsection{Formulation Spaces}
\label{subsec:function-spaces}
\begin{assume}\label{as:func}
Let $n$ be a positive integer and $\Omega\subset\mathbb R^n$ be a bounded Lipschitz
domain. Throughout this paper, $W$ is a Banach space of real-valued functions on $\Omega$ with a continuous embedding $W \hookrightarrow C^1(\overline{\Omega})$. 
Thus there exists $C_\Omega>0$ such that
\begin{equation}
    \|f\|_{C^1(\overline{\Omega})}
    \le
    C_\Omega\|f\|_W
    \qquad (f\in W).
\end{equation}
\end{assume}
For standard background on functional
analysis, we refer the reader to Brezis~\cite{Brezis2011}.
\begin{ex}
The following are typical examples of spaces satisfying the Assumption~\ref{as:func}.

\begin{itemize}
\item $W=C^1(\overline\Omega)$ with its standard norm.

\item Sobolev spaces $W^{s,p}(\Omega)$ satisfying
\begin{equation}
    s>1+\frac{n}{p},
\end{equation}
under the usual Sobolev embedding assumptions. In this case,
$W^{s,p}(\Omega)\hookrightarrow C^1(\overline\Omega)$.

\item For $d\ge 1$, the polynomial space
\begin{equation}
    W=\mathbb R[x_1,\dots,x_n]_{\le d}
\end{equation}
of polynomials of degree at most $d$, equipped with any norm.
\end{itemize}
\end{ex}

In this paper, the notation related to symmetric products follows that of \cite{borsuk1931symmetric}.

We consider constraint systems on $\Omega$ of the form
\begin{equation}
    h_{\alpha}(x)=0,\quad \alpha=1,\ldots,p,
    \qquad
    g_i(x)\le 0,\quad i=1,\ldots,q,
\end{equation}
where $h_\alpha\in W,\ g_i\in W$.
The functions $h_\alpha$ are equality constraints, while the functions
$g_i$ are inequality constraints.

The ordering of equality constraints and the ordering of inequality
constraints carry no intrinsic meaning. Thus, two constraint systems that
differ only by a permutation of equality constraints or by a permutation of
inequality constraints are identified. However, multiplicities are retained, 
so repeated constraints are treated as separate instances of the same constraint.

For fixed nonnegative integers $p,q$, the $(p,q)$-stratum of constraint
formulations is defined by
\begin{equation}
    \mathcal C_{p,q}
    =
    \operatorname{Sym}^p(W)\times \operatorname{Sym}^q(W).
\end{equation}
Thus a point $P\in\mathcal C_{p,q}$ is an ordered pair
\begin{equation}
    P=(H_P,G_P),
\end{equation}
where $H_P\in\operatorname{Sym}^p(W),\ G_P\in\operatorname{Sym}^q(W)$.
Here $H_P$ is the collection of equality constraints and $G_P$ is the
collection of inequality constraints.

In this paper, we work on the quotient space but choose labelled representatives when needed. 
\subsection{Feasible sets, active sets, and LICQ}
\label{subsec:feasible-active-licq}

Let $P\in\mathcal C_{p,q}$ be written as
\begin{equation}
    P=(H_P,G_P),
\end{equation}
where
\begin{equation}
    H_P
    =
    \overline{(h_{P,1},\ldots,h_{P,p})},
    \qquad
    G_P
    =
    \overline{(g_{P,1},\ldots,g_{P,q})}.
\end{equation}
The feasible set associated with $P$ is denoted by
\begin{equation}
    K(P)
    =
    \{x\in \Omega \mid h(x)=0\ \text{for all }h\in H_P,
        \quad
        g(x)\le 0\ \text{for all }g\in G_P\}.
\end{equation}
The quantifiers over $H_P$ and $G_P$ are understood in the sense of unordered collections; repeated constraints impose the same scalar condition more than once
as part of the formulation, although they do not change the underlying
feasible set.

For $x\in K(P)$, and after choosing a labelled representative of $G_P$,
the active inequality index set is
\begin{equation}
    I_P(x)
    =
    \{i\in\{1,\ldots,q\} \mid g_{P,i}(x)=0\}.
\end{equation}

Equality constraints are always active. Therefore, for a prescribed active
index set $I\subset\{1,\ldots,q\}$, we define the active constraint map
using a labelled representative by
\begin{equation}
    c_{P,I}
    =
    (h_{P,1},\ldots,h_{P,p},g_{P,i_1},\dots,g_{P,i_{|I|}}).
\end{equation}
If $I=\{i_1,\ldots,i_{|I|}\}$, its Jacobian at $x$ is the
$(p+|I|)\times n$ matrix
\begin{equation}
    D c_{P,I}(x)
    =
    \begin{bmatrix}
        D h_{P,1}(x)\\
        \vdots\\
        D h_{P,p}(x)\\
        D g_{P,i_1}(x)\\
        \vdots\\
        D g_{P,i_{|I|}}(x)
    \end{bmatrix}.
\end{equation}
Here each $D h(x)$ and $D g(x)$ are regarded as a row vector in
$\mathbb R^n$.
\begin{dfn}
The linear independence constraint qualification (LICQ) holds
for $P$ at $x\in K(P)$ if
\begin{equation}
    \operatorname{rank} D c_{P,I_P(x)}(x) = p+|I_P(x)|.
\end{equation}
\end{dfn}

This condition is independent of the chosen labelled representative, since
permuting constraints only permutes the rows of the Jacobian matrix. In
particular, if $p+|I_P(x)|>n$, then LICQ necessarily fails at $x$.

All definitions and results below are invariant under permutations of the
equality constraints and under permutations of the inequality constraints.

\section{Distances on Formulation Spaces and Basic Properties}
\label{sec:distance}
In this section, distances are defined only on each fixed stratum
$\mathcal C_{p,q}$. 

Since equality and inequality constraints have different roles, they are
matched only with constraints of the same type. Thus the distance on a fixed
stratum is defined componentwise with respect to $P=(H_P,G_P)\in\mathcal C_{p,q}$.

\subsection{Distances on Formulation Spaces}
We first define the matching distance on a fixed symmetric product. Let
$r\in\mathbb N_0$. For $A=\overline{(\phi_1,\ldots,\phi_r)},\ B=\overline{(\psi_1,\ldots,\psi_r)}\in \operatorname{Sym}^r(W)$,
we define
\begin{equation}
\label{eq:sym-matching-distance}
    d_{\operatorname{Sym}^r(W)}(A,B)
    =
    \min_{\sigma\in S_r}
    \left(
        \sum_{i=1}^{r}
        \|\phi_i-\psi_{\sigma(i)}\|_W^2
    \right)^{1/2}.
\end{equation}
When $r=0$, $\operatorname{Sym}^0(W)=\{\varnothing\}$, and the distance
is identically zero.

\begin{prop}
\label{prop:sym-matching-metric}
For any $r\in\mathbb N_0$, the function $d_{\operatorname{Sym}^r(W)}$
defined by \eqref{eq:sym-matching-distance} is a metric on
$\operatorname{Sym}^r(W)$.
\end{prop}
\begin{proof}
The distance defined in \ref{eq:sym-matching-distance} is the quotient distance induced by this
finite isometric group action.
\end{proof}
By Proposition~\ref{prop:sym-matching-metric}, we define a distance on $\mathcal C_{p,q}$

\begin{dfn}
\label{def:constraint-matching-distance}
Let $P=(H_P,G_P),\ Q=(H_Q,G_Q) \in \mathcal C_{p,q}$. 
Choose labelled representatives
\begin{equation}
    H_P=(h_{P,1},\ldots,h_{P,p}),
    \qquad
    H_Q=(h_{Q,1},\ldots,h_{Q,p}),
\end{equation}
and
\begin{equation}
    G_P=(g_{P,1},\ldots,g_{P,q}),
    \qquad
    G_Q=(g_{Q,1},\ldots,g_{Q,q}).
\end{equation}
We define
\begin{equation}
\label{eq:constraint-distance-product}
    d_{p,q}(P,Q)^2 = d_{\operatorname{Sym}^p(W)}(H_P,H_Q)^2 + d_{\operatorname{Sym}^q(W)}(G_P,G_Q)^2.
\end{equation}
\end{dfn}
Thus equality constraints are matched only with equality constraints, and
inequality constraints are matched only with inequality constraints. The cost
of matching two constraint functions is their $W$-norm distance, and the
total cost is computed by an optimal matching.
\begin{prop}
\label{prop:constraint-stratum-metric}
For any $p,q\in\mathbb N_0$, the tuple $(\mathcal C_{p,q},d_{p,q})$
is a metric space.
\end{prop}
\begin{proof}
The space $\mathcal C_{p,q}$ is the product
\begin{equation}
    \operatorname{Sym}^p(W)\times\operatorname{Sym}^q(W).
\end{equation}
By Proposition~\ref{prop:sym-matching-metric}, each factor is a metric
space. Hence $d_{p,q}$ is a metric.
\end{proof}

We next introduce the operation of adding one constraint function. This
operation moves a formulation from one stratum to an adjacent stratum.

For $\phi\in W$, define
\begin{equation}
    \operatorname{add}^{=}_{\phi}:
    \mathcal C_{p,q}\to\mathcal C_{p+1,q}
\end{equation}
and
\begin{equation}
    \operatorname{add}^{\le}_{\phi}:
    \mathcal C_{p,q}\to\mathcal C_{p,q+1}
\end{equation}
by
\begin{equation}
\label{eq:add-equality-constraint}
    \operatorname{add}^{=}_{\phi}((H_P,G_P))
    =
    \left(
        \overline{(h_{P,1},\dots,h_{P,p},\phi)},
        G_P
    \right),
\end{equation}
and
\begin{equation}
\label{eq:add-inequality-constraint}
    \operatorname{add}^{\le}_{\phi}((H_P,G_P))
    =
    \left(
        H_P,
        \overline{(g_{P,1},\dots,g_{P,q},\phi)}
    \right).
\end{equation}

We write
\begin{equation}
    P^{\phi=}
    =
    \operatorname{add}^{=}_{\phi}(P),
    \qquad
    P^{\phi\le}
    =
    \operatorname{add}^{\le}_{\phi}(P).
\end{equation}
Thus $P^{\phi=}$ is obtained from $P$ by adding $\phi=0$ as an equality
constraint, while $P^{\phi\le}$ is obtained from $P$ by adding $\phi\le 0$ as an inequality constraint.

\begin{prop}
\label{prop:add-constraint-lipschitz}
For $\phi\in W$, the maps
\begin{equation}
    \operatorname{add}^{=}_{\phi}:
    \mathcal C_{p,q}\to\mathcal C_{p+1,q},
    \qquad
    \operatorname{add}^{\le}_{\phi}:
    \mathcal C_{p,q}\to\mathcal C_{p,q+1}
\end{equation}
are $1$-Lipschitz with respect to the matching distances, that is,
for all $P,Q\in\mathcal C_{p,q}$,
\begin{equation}
    d_{p+1,q}(P^{\phi=},Q^{\phi=})
    \le
    d_{p,q}(P,Q),
\end{equation}
and
\begin{equation}
    d_{p,q+1}(P^{\phi\le},Q^{\phi\le})
    \le
    d_{p,q}(P,Q).
\end{equation}
\end{prop}

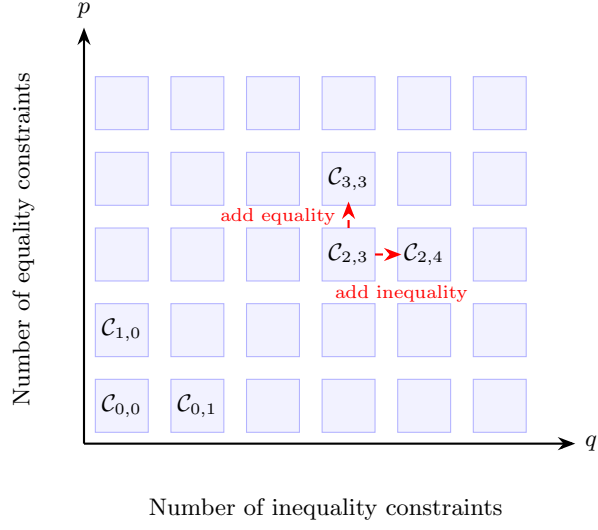
\begin{figure}[ht]
\centering
\begin{tikzpicture}[scale=1.0, font=\small]
    % Axes
    \draw[-{Stealth}, thick] (0,0) -- (6.5,0) node[right] {$q$};
    \draw[-{Stealth}, thick] (0,0) -- (0,5.5) node[above] {$p$};
    \node[below] at (3.2,-0.6) {Number of inequality constraints};
    \node[rotate=90] at (-0.8,2.7) {Number of equality constraints};

    % Grid strata
    \foreach \p in {0,1,2,3,4} {
        \foreach \q in {0,1,2,3,4,5} {
            \draw[fill=blue!5, draw=blue!30]
                (\q+0.15,\p+0.15) rectangle (\q+0.85,\p+0.85);
        }
    }

    % Labels
    \node at (0.5,0.5) {$\mathcal C_{0,0}$};
    \node at (1.5,0.5) {$\mathcal C_{0,1}$};
    \node at (0.5,1.5) {$\mathcal C_{1,0}$};
    \node at (3.5,2.5) {$\mathcal C_{2,3}$};
    \node at (4.5,2.5) {$\mathcal C_{2,4}$};
    \node at (3.5,3.5) {$\mathcal C_{3,3}$};

    % Horizontal arrow: add inequality
    \draw[-{Stealth}, thick, red, dashed] (3.85,2.5) -- (4.2,2.5);
    \node[red, above] at (4.2,1.75) {\scriptsize add inequality};

    % Vertical arrow: add equality
    \draw[-{Stealth}, thick, red, dashed] (3.5,2.85) -- (3.5,3.2);
    \node[red, left] at (3.45,3.0) {\scriptsize add equality};
\end{tikzpicture}
\caption{Stratification of the constraint formulation space by the number
$p$ of equality constraints and the number $q$ of inequality constraints.}
\label{fig:pq-stratification}
\end{figure}

\subsection{Basic Properties}
\label{sec:stab}

In this section, we record elementary stability properties of the matching
metric on $\mathcal C_{p,q}$. The purpose is to clarify
what is stable under small perturbations of a constraint formulation.
Because equality constraints are present, feasibility of a fixed point is
not, in general, preserved under arbitrary perturbations. What is preserved
is almost feasibility, together with the strict inactivity of inactive
inequality constraints.

Throughout this section, we fix $p,q\in\mathbb N_0$ and work on the
stratum
\begin{equation}
    \mathcal C_{p,q}
    =
    \operatorname{Sym}^p(W)\times\operatorname{Sym}^q(W).
\end{equation}
We write
\begin{equation}
    P=(H_P,G_P)\in\mathcal C_{p,q},
\end{equation}
where
\begin{equation}
    H_P=\overline{(h_{P,1},\ldots,h_{P,p})},
    \qquad
    G_P=\overline{(g_{P,1},\ldots,g_{P,q})},
\end{equation}
after choosing a labelled representative. All statements below are invariant
under permutations of the equality constraints and under permutations of the
inequality constraints.

\begin{prop}
\label{prop:almost-feasible}
Let $P=(H_P,G_P)\in\mathcal C_{p,q}$ and $x\in K(P)$. Then, for every
$\varepsilon>0$, there exists $\delta>0$ such that every
$Q=(H_Q,G_Q)\in\mathcal C_{p,q}$ with $d_{p,q}(P,Q)<\delta$
admits labelled representatives
\begin{equation}
    H_P=(h_{P,1},\ldots,h_{P,p}),
    \qquad
    H_Q=(h_{Q,1},\ldots,h_{Q,p}),
\end{equation}
and
\begin{equation}
    G_P=(g_{P,1},\ldots,g_{P,q}),
    \qquad
    G_Q=(g_{Q,1},\ldots,g_{Q,q}),
\end{equation}
such that
\begin{equation}
    |h_{Q,\alpha}(x)|\le \varepsilon,
    \qquad
    \alpha=1,\ldots,p,
\end{equation}
and
\begin{equation}
    g_{Q,j}(x)\le \varepsilon,
    \qquad
    j=1,\ldots,q.
\end{equation}
\end{prop}

\begin{proof}
By the definition of the matching distance $d_{p,q}$, there exist
permutations of the equality and inequality constraints of $Q$ such that,
after relabelling,
\begin{equation}
    \sum_{\alpha=1}^p
    \|h_{P,\alpha}-h_{Q,\alpha}\|_W^2
    +
    \sum_{i=1}^q
    \|g_{P,i}-g_{Q,i}\|_W^2
    =
    d_{p,q}(P,Q)^2.
\end{equation}
Hence, for each $\alpha$,
\begin{equation}
    \|h_{P,\alpha}-h_{Q,\alpha}\|_W
    \le d_{p,q}(P,Q).
\end{equation}
By the continuous embedding $W\hookrightarrow C^0(\overline\Omega)$,
\begin{equation}
    |h_{P,\alpha}(x)-h_{Q,\alpha}(x)|
    \le C_{\Omega} d_{p,q}(P,Q).
\end{equation}
Since $x\in K(P)$, we have $h_{P,\alpha}(x)=0$ for all $h_{P,\alpha}\in H_P$.
Therefore we have
\begin{equation}
    |h_{Q,\alpha}(x)|
    \le C_{\Omega} d_{p,q}(P,Q).
\end{equation}
Taking $\delta<\varepsilon/C_{\Omega}$, we have
\begin{equation}
    |h_{Q,\alpha}(x)|\le \varepsilon.
\end{equation}
Similarly, since $g_{P,j}(x)\le 0$ for all $j=1,\ldots,q$, we obtain
\begin{equation}
    g_{Q,j}(x)\le \varepsilon,
    \qquad j=1,\ldots,q.
\end{equation}
\end{proof}
\begin{prop}
\label{cor:inactive-stability}
Let $P=(H_P, G_P)\in\mathcal C_{p,q}$ and $x\in K(P)$.
Suppose that $\gamma=\min_{j\notin I_P(x)}\{-g_{P,j}(x)\}>0$ with the convention that $\gamma=+\infty$ if $I_P(x)=\{1,\ldots,q\}$.
Then there exists $\delta>0$ such that every
$Q\in\mathcal C_{p,q}$ with $d_{p,q}(P,Q)<\delta$ admits a relabelling for which
\begin{equation}
    g_{Q,j}(x)<0
    \qquad
    \text{for all }j\notin I_P(x).
\end{equation}

\end{prop}

\begin{proof}
After relabelling $Q$ by an optimal matching as in
Proposition~\ref{prop:almost-feasible}, we have
\begin{equation}
    |g_{P,j}(x)-g_{Q,j}(x)|
    \le C_{\Omega} d_{p,q}(P,Q)
\end{equation}
for all $j=1,\ldots,q$. For $j\notin I_P(x)$,
\begin{equation}
    g_{P,j}(x)\le -\gamma.
\end{equation}
Therefore
\begin{equation}
    g_{Q,j}(x)
    \le
    g_{P,j}(x)+C_{\Omega}d_{p,q}(P,Q)
    \le
    -\gamma+C_{\Omega}d_{p,q}(P,Q).
\end{equation}
If $d_{p,q}(P,Q)<\gamma/C_{\Omega}$, then $g_{Q,j}(x)<0$ for all $j\notin  I_P(x)$.
\end{proof}

\section{The LICQ Radius and Jet Formulation Spaces}
\label{sec:radius}
In this section, we introduce the LICQ radius and the Jet space associated with the formulation space $\mathcal{C}_{p,q}$. 

The LICQ radius serves as a measure of the robustness of the LICQ condition for a given problem and its solution. 
While the direct computation of this radius on $\mathcal C_{p,q}$ is generally
difficult from the definition, LICQ is fundamentally a local property of
optimization problems. 
By employing the jet space, which extracts only the local information of functions and their gradients at the reference point, the calculation of the LICQ radius can be reduced to matrix computations. 
In particular, we show that the LICQ radius on the Jet space coincides with the radius on the original space $\mathcal{C}_{p,q}$.
An element $P$ of $\mathcal C_{p,q}$ will be written as
\begin{equation}
    P=(H_P,G_P),
\end{equation}
where
\begin{equation}
    H_P=\overline{(h_1,\ldots,h_p)},
    \qquad
    G_P=\overline{(g_1,\ldots,g_q)},
\end{equation}
with the understanding that the equality and inequality families are unordered
within their own types.

We define the pointwise LICQ radius as the distance to LICQ failure
\begin{dfn}\label{def:LICQrad}
\begin{equation}
    r_{\mathrm{LICQ},W}^{p,q}(P,x)
    =
    \inf\left\{
        d_{p,q}(P,Q)
        \mid
        x\in K(Q)
        \text{ and } Q \text{ does not satisfy LICQ at } x
    \right\}.
\end{equation}
If the set in the infimum is empty, the value is understood to be $+\infty$.
\end{dfn}
\subsection{Jet spaces and quotient norms}\label{subsec:Jet}
We use a simplified, pointwise version of the theory of first-order jets; for
background on jet bundles, see \cite{Kolar1993} or \cite{Saunders1989}. In the present paper, we
only need the jet space over a fixed point $x$, rather than the full jet
bundle over the domain.

In the remainder of the paper, we derive explicit formulas for the LICQ
radius, which measures the robustness of LICQ under perturbations of the
constraint functions. Since the active status of an inequality constraint and
the linear independence condition in LICQ are determined, at a fixed point
$x$, by the values and gradients of the constraint functions, it is sufficient
to retain only their first-order Taylor data at $x$. Thus, instead of working
with the entire function space $W$, we pass to the first-order jet space at
$x$.

Concretely, the first-order jet records the information appearing in the
first-order Taylor polynomial
\begin{equation}
    f(x)+\nabla f(x)^\top (y-x).
\end{equation}
The simplest model is the space of affine functions
\begin{equation}
    f(y)=a+\xi^\top(y-x),
\end{equation}
where the pair $(a,\xi)$ is exactly the first-order jet at $x$. The quotient
jet space introduced below formalizes this idea for a general function space
$W$.
\subsubsection{Jet spaces}\label{subsubsec:Jet}
For a fixed point $x\in\overline\Omega$, let
\begin{equation}
    J_x=J_x^1(W)=W/\mathcal N_x^1,
    \qquad
    \mathcal N_x^1
    =
    \{f\in W \mid f(x)=0,\ \nabla f(x)=0\}.
\end{equation}
Since the embedding $W\hookrightarrow C^1(\overline\Omega)$ is continuous,
$\mathcal N_x^1$ is closed in $W$.
We endow $J_x$ with the quotient norm
\begin{equation}
    \|\Pi_x^1(f)\|_{J_x}
    =
    \inf\{\|u\|_W \mid u-f\in\mathcal N_x^1\}.
\end{equation}

Set
\begin{equation}
    \mathcal C_{p,q}^J(x)
    =
    \operatorname{Sym}^p(J_x)\times \operatorname{Sym}^q(J_x).
\end{equation}
We denote by
\begin{equation}
    \Pi_x^1:W\to J_x
\end{equation}
the natural quotient map and by
\begin{equation}
    \Pi_x^{1,(p,q)}:W^p\times W^q\to J_x^p\times J_x^q
\end{equation}
its componentwise extension.

Let
\begin{equation}
    T_{p,q}:W^p\times W^q\to \mathcal C_{p,q}
\end{equation}
be the natural quotient map to the typed symmetric product, and let
\begin{equation}
    T_{p,q}^J:J_x^p\times J_x^q\to \mathcal C_{p,q}^J(x)
\end{equation}
be the corresponding quotient map on the jet side.

Since $\Pi_x^{1,(p,q)}$ is equivariant with respect to the action of
$S_p\times S_q$, it induces a well-defined map
\begin{equation}
    \widehat\Pi_x^1:\mathcal C_{p,q}\to \mathcal C_{p,q}^J(x)
\end{equation}
such that the following diagram commutes:
\begin{equation}
\begin{tikzcd}
W^p\times W^q
    \arrow[r, "\Pi_x^{1,(p,q)}"]
    \arrow[d, "T_{p,q}"']
&
J_x^p\times J_x^q
    \arrow[d, "T_{p,q}^J"]
\\
\mathcal C_{p,q}
    \arrow[r, "\widehat\Pi_x^1"']
&
\mathcal C_{p,q}^J(x).
\end{tikzcd}
\end{equation}
Hence, the induced distance on $\mathcal C_{p,q}^J(x)$ is defined by
\begin{equation}
    \widetilde d_{p,q}(Z_1,Z_2)
    =
    \inf\left\{
        d_{p,q}(P,Q)
        \ \middle|\
        \widehat\Pi_x^1(P)=Z_1,\ 
        \widehat\Pi_x^1(Q)=Z_2
    \right\}.
\end{equation}
\begin{lem}
The quotient norm on $J_x$ induces a matching metric
$\widetilde d_{p,q}$ on $\mathcal C_{p,q}^J(x)$.
With this metric, $\mathcal C_{p,q}^J(x)$ is a metric space. Moreover, for every
$P\in \mathcal C_{p,q}$ and every $Z\in \mathcal C_{p,q}^J(x)$,
\begin{equation}
    \widetilde d_{p,q}(\widehat\Pi_x^1(P),Z)
    =
    \inf\left\{
        d_{p,q}(P,Q)
        \mid
        \widehat\Pi_x^1(Q)=Z
    \right\}.
\end{equation}
\end{lem}
\begin{proof}
Let
\begin{equation}
    \mathcal F_Z=\{Q\in\mathcal C_{p,q}\mid \widehat\Pi_x^1(Q)=Z\}.
\end{equation}
By the definition of $\widetilde d_{p,q}$,
\begin{equation}
\widetilde d_{p,q}(\widehat\Pi_x^1(P),Z)
=
\inf\left\{
    d_{p,q}(P',Q)
    \ \middle|\
    \widehat\Pi_x^1(P')=\widehat\Pi_x^1(P),\ 
    Q\in\mathcal F_Z
\right\}.
\end{equation}
Since $P$ itself satisfies $\widehat\Pi_x^1(P)=\widehat\Pi_x^1(P)$, we immediately
obtain
\begin{equation}
    \widetilde d_{p,q}(\widehat\Pi_x^1(P),Z)
    \le
    \inf_{Q\in\mathcal F_Z}d_{p,q}(P,Q).
\end{equation}

It remains to prove the reverse inequality. Take arbitrary
$P'\in\mathcal C_{p,q}$ and $Q\in\mathcal F_Z$ with
\begin{equation}
    \widehat\Pi_x^1(P')=\widehat\Pi_x^1(P).
\end{equation}
Choose labelled representatives, after relabelling within equality and
inequality types, so that
\begin{equation}
    P=(h_1,\ldots,h_p, g_1,\ldots,g_q),
    \qquad
    P'=(h'_1,\ldots,h'_p, g'_1,\ldots,g'_q),
\end{equation}
and
\begin{equation}
    \Pi_x^1(h'_i)=\Pi_x^1(h_i),
    \qquad
    \Pi_x^1(g'_j)=\Pi_x^1(g_j).
\end{equation}
Choose a labelled representative
\begin{equation}
    Q=(\tilde h_1,\ldots,\tilde h_p, \tilde g_1,\ldots,\tilde g_q)
\end{equation}
realizing, or $\varepsilon$-realizing, the matching distance from $P'$ to
$Q$. Define
\begin{equation}
    \tilde h_i^*
    =
    \tilde h_i-(h'_i-h_i),
    \qquad
    \tilde g_j^*
    =
    \tilde g_j-(g'_j-g_j),
\end{equation}
and set
\begin{equation}
    Q^*=(\tilde h_1^*,\ldots,\tilde h_p^*, \tilde g_1^*,\ldots,\tilde g_q^*).
\end{equation}
Since
\begin{equation}
    \Pi_x^1(h'_i-h_i)=0,
    \qquad
    \Pi_x^1(g'_j-g_j)=0,
\end{equation}
we have
\begin{equation}
    \widehat\Pi_x^1(Q^*)=\widehat\Pi_x^1(Q)=Z.
\end{equation}
Moreover,
\begin{equation}
    h_i-\tilde h_i^*
    =
    h'_i-\tilde h_i,
    \qquad
    g_j-\tilde g_j^*
    =
    g'_j-\tilde g_j.
\end{equation}
Therefore the labelled cost from $P$ to $Q^*$ is the same as the labelled
cost from $P'$ to $Q$. Hence
\begin{equation}
    d_{p,q}(P,Q^*)\le d_{p,q}(P',Q).
\end{equation}
Since $Q^*\in\mathcal F_Z$, it follows that
\begin{equation}
    \inf_{R\in\mathcal F_Z}d_{p,q}(P,R)
    \le
    d_{p,q}(P',Q).
\end{equation}
Taking the infimum over all admissible $P'$ and $Q$ gives
\begin{equation}
    \inf_{R\in\mathcal F_Z}d_{p,q}(P,R)
    \le
    \widetilde d_{p,q}(\widehat\Pi_x^1(P),Z).
\end{equation}
Combining the two inequalities proves
\begin{equation}
    \widetilde d_{p,q}(\widehat\Pi_x^1(P),Z)
    =
    \inf\left\{
        d_{p,q}(P,Q)
        \ \middle|\
        \widehat\Pi_x^1(Q)=Z
    \right\}.
\end{equation}
\end{proof}
For $Z=(H_Z,G_Z)\in \mathcal C_{p,q}^J(x)$,
choose representatives
\begin{equation}
    H_Z=\overline{((a_1,\xi_1),\dots, (a_p,\xi_p))},
    \qquad
    G_Z=\overline{((c_1,\zeta_1),\dots,(c_q,\zeta_q))}.
\end{equation}
We say that $Z$ is feasible at $x$ if
\begin{equation}
    a_\alpha=0
    \quad \text{for all } \alpha=1,\ldots,p,
    \qquad
    c_i\le 0
    \quad \text{for all } i=1,\ldots,q.
\end{equation}
If $Z$ is feasible at $x$, its active inequality functions at $x$ are those
$(c_i,\zeta_i)\in G_Z$ satisfying $c_i=0$.

We say that a feasible jet formulation $Z$ satisfies LICQ at $x$ if the family
\begin{equation}
    \{\xi_\alpha\}_{\alpha=1}^p
    \cup
    \{\zeta_i \mid (c_i,\zeta_i)\in G_Z,\ c_i=0\}
\end{equation}
is linearly independent. This definition is independent of the chosen ordering
of the representatives.

The pointwise LICQ radius on $\mathcal C_{p,q}^J(x)$ is defined by
\begin{equation}
    r_{\mathrm{LICQ},J}^{p,q}(Z_1,x)
    =
    \inf\left\{
        \widetilde d_{p,q}(Z_1,Z_2)
        \mid
        Z_2 \text{ is feasible at }x
        \text{ and } Z_2 \text{ does not satisfy LICQ at }x
    \right\}.
\end{equation}
If the set in the infimum is empty, the value is understood to be $+\infty$.

Set 
\begin{equation}
    \mathcal B_{p,q}^W(x)
    =
    \{Q\in\mathcal C_{p,q}
      \mid
      x\in K(Q)
      \text{ and } Q \text{ does not satisfy LICQ at }x\}
\end{equation}
and
\begin{equation}
    \mathcal B_{p,q}^J(x)
    =
    \{Z\in\mathcal C_{p,q}^J(x)
      \mid
      Z \text{ is feasible at }x
      \text{ and } Z \text{ does not satisfy LICQ at }x\}.
\end{equation}

\begin{prop}
\label{prop:equivalence-radii}
For every $P\in \mathcal C_{p,q}$ and every $x\in K(P)$,
\begin{equation}
    r_{\mathrm{LICQ},W}^{p,q}(P,x)
    =
    r_{\mathrm{LICQ},J}^{p,q}(\widehat\Pi_x^1(P),x).
\end{equation}
\end{prop}

\begin{proof}
By the definitions of feasibility and LICQ on the jet formulation space,
feasibility at $x$ and LICQ at $x$ depend only on the first-order jets.
Hence we have
\begin{equation}
    \mathcal B_{p,q}^W(x)
    =
    (\widehat\Pi_x^1)^{-1}\bigl(\mathcal B_{p,q}^J(x)\bigr).
\end{equation}
Therefore,
\begin{equation}
\begin{aligned}
r_{\mathrm{LICQ},W}^{p,q}(P,x)
&=
\inf_{Q\in \mathcal B_{p,q}^W(x)}
d_{p,q}(P,Q) \\
&=
\inf_{Z\in \mathcal B_{p,q}^J(x)}
\inf_{\widehat\Pi_x^1(Q)=Z}
d_{p,q}(P,Q) \\
&=
\inf_{Z\in \mathcal B_{p,q}^J(x)}
\widetilde d_{p,q}(\widehat\Pi_x^1(P),Z) \\
&=
r_{\mathrm{LICQ},J}^{p,q}(\widehat\Pi_x^1(P),x).
\end{aligned}
\end{equation}
This proves the claim.
\end{proof}

\subsubsection{Quotient norm}
\label{subssubec:quotient_norm}
Throughout the remainder of this paper, we fix a point $x \in \Omega$ and work under the following functional analytic setting.

\begin{assume}
\label{assum:jet-realizability}
We assume that:
\begin{enumerate}
    \item The function space $W$ is a Hilbert space. 
    \item The first-order evaluation map $E_x: W \to \mathbb{R}^{1+n}$ defined by $E_x(f) = (f(x), \nabla f(x))$ is a bounded linear operator, satisfying 
    \begin{equation}
        \|E_x(f)\|_{\mathbb{R}^{1+n}} \le C_{\Omega} \|f\|_W \quad \text{for all } f \in W.
    \end{equation}
    \item The map $E_x$ is surjective.
\end{enumerate}
\end{assume}

Under Assumption~\ref{assum:jet-realizability}, the surjectivity of $E_x$ ensures that any prescribed value and gradient at $x$ can be realized by some function in $W$.

 Moreover,
\begin{equation}
    \mathcal N_x^1=\ker E_x.
\end{equation}
Hence $E_x$ induces a linear isomorphism
\begin{equation}
    \overline E_x:J_x=W/\mathcal N_x^1\to \mathbb R^{1+n},
    \qquad
    \overline E_x(\Pi_x^1(f))=(f(x),\nabla f(x)).
\end{equation}
We identify $J_x$ with $\mathbb R\times\mathbb R^n$ through this isomorphism.

Under this identification, the quotient norm on $J_x$ is given by
\begin{equation}\label{eq:quotient-jet-norm-min}
    \|(a,\xi)\|_{J_x}^2
    =
    \min\left\{
        \|f\|_W^2
        \mid
        f(x)=a,\ \nabla f(x)=\xi
    \right\}.
\end{equation}
Indeed, the quotient norm is defined by
\begin{equation}
    \|\Pi_x^1(f)\|_{J_x}
    =
    \inf\left\{
        \|f+\varphi\|_W
        \mid
        \varphi\in\mathcal N_x^1
    \right\},
\end{equation}
and the affine space
\begin{equation}
    \{u\in W\mid u(x)=a,\ \nabla u(x)=\xi\}
\end{equation}
is exactly the equivalence class of any representative with first-order jet
$(a,\xi)$. Since $\mathcal N_x^1$ is a closed subspace of the Hilbert space
$W$, the above infimum is attained by the orthogonal projection onto
$(\mathcal N_x^1)^\perp$.

Consequently, \eqref{eq:quotient-jet-norm-min} defines the square of a Hilbert
norm on the finite-dimensional space $\mathbb R\times\mathbb R^n$. Hence there
exists a positive definite matrix
\begin{equation}
    H_x\in\mathbb R^{(1+n)\times(1+n)}
\end{equation}
such that
\begin{equation}
    \|(a,\xi)\|_{J_x}^2
    =
    \begin{pmatrix}
        a\\ \xi
    \end{pmatrix}^{\!\top}
    H_x
    \begin{pmatrix}
        a\\ \xi
    \end{pmatrix}.
\end{equation}
Writing $H_x$ in block form as
\begin{equation}
    H_x=
    \begin{pmatrix}
        \alpha_x & b_x^\top\\
        b_x & D_x
    \end{pmatrix},
\end{equation}
where $\alpha_x\in\mathbb R$, $b_x\in\mathbb R^n$, and
$D_x\in\mathbb R^{n\times n}$, we have $D_x\succ0$. Completing the square gives
\begin{equation}
\begin{aligned}
    \|(a,\xi)\|_{J_x}^2
    &=
    \alpha_x a^2+2a\,b_x^\top\xi+\xi^\top D_x\xi  \\
    &=
    \left(\alpha_x-b_x^\top D_x^{-1}b_x\right)a^2
    +
    \left(\xi+aD_x^{-1}b_x\right)^\top
    D_x
    \left(\xi+aD_x^{-1}b_x\right).
\end{aligned}
\end{equation}
Therefore, setting
\begin{equation}
    C_x=D_x^{-1},
    \qquad
    v_x=D_x^{-1}b_x,
    \qquad
    \mu_x=\alpha_x-b_x^\top D_x^{-1}b_x,
\end{equation}
we obtain
\begin{equation}\label{eq:jet-norm-block-representation}
    \|(a,\xi)\|_{J_x}^2
    =
    \mu_x a^2+\|\xi+a v_x\|_{C_x^{-1}}^2.
\end{equation}
where $\|y\|_{C_x^{-1}}^2=y^\top C_x^{-1}y$.
Since $H_x\succ0$, its Schur complement satisfies
\begin{equation}
    \mu_x=\alpha_x-b_x^\top D_x^{-1}b_x>0,
\end{equation}
and $C_x$ is positive definite.

\begin{prop}
Let $Z=(H_Z,G_Z),\ Y=(H_Y,G_Y) \in \mathcal C_{p,q}^J(x)$, and choose representatives
\begin{equation}
    H_Z=\overline{((a_1,\xi_1),\dots,(a_p,\xi_p))},
    \quad
    H_Y=\overline{((b_1,\eta_1),\dots,(b_p,\eta_p))},
\end{equation}
\begin{equation}
    G_Z=\overline{((c_1,\zeta_1),\dots,(a_q,\zeta_q))},
    \quad
    G_Y=\overline{((d_1,\theta_1),\dots,(d_q,\theta_q))}. 
\end{equation}
Then we have
\begin{equation}
\begin{aligned}
\widetilde d_{p,q}(Z,Y)^2
&=
\min_{\sigma\in S_p}
\sum_{\alpha=1}^p
\left[
    \mu_x(a_\alpha-b_{\sigma(\alpha)})^2
    +
    \left\|
        (\xi_\alpha-\eta_{\sigma(\alpha)})
        +(a_\alpha-b_{\sigma(\alpha)})v_x
    \right\|_{C_x^{-1}}^2
\right]
\\
&\quad+
\min_{\tau\in S_q}
\sum_{i=1}^q
\left[
    \mu_x(c_i-d_{\tau(i)})^2
    +
    \left\|
        (\zeta_i-\theta_{\tau(i)})
        +(c_i-d_{\tau(i)})v_x
    \right\|_{C_x^{-1}}^2
\right].
\end{aligned}
\end{equation}
\end{prop}
\begin{proof}
This follows directly from the definition of the matching metric on
$\mathcal C_{p,q}^J(x)$.
\end{proof}
\section{Explicit formulas for LICQ radii}
\label{sec:Explicit}
In this section we derive an explicit formula for the pointwise LICQ
radius.  The formula is obtained by decomposing the perturbation needed
to destroy LICQ at a fixed feasible point into two contributions.  The
first one is the cost of changing the values of inequality constraints so
that a prescribed subset of them becomes active at the point.  The second
one is the cost of perturbing the corresponding gradients so that the
gradients of the equality constraints together with those active
inequality constraints become linearly dependent. 
\begin{dfn}

Let $P=(H_P,G_P)\in\mathcal C_{p,q}$ and $x\in K(P)$. 
Choose labelled representatives
\begin{equation}
    H_P=(h_1,\ldots,h_p),
    \qquad
    G_P=(g_1,\ldots,g_q).
\end{equation}
\begin{enumerate}
    \item For an index set $I\subseteq\{1,\ldots,q\}$, define the slack vector
\begin{equation}
    s_I(P,x)=(-g_i(x))_{i\in I}\in\mathbb R^{|I|}.
\end{equation}
We also define the extended slack vector
\begin{equation}
    \bar s_I(P,x)
    =
    (\underbrace{0,\ldots,0}_{p\ \mathrm{times}},\,s_I(P,x))
    \in\mathbb R^{p+|I|}.
\end{equation}
\item Let $A_I(P,x)\in\mathbb R^{(p+|I|)\times n}$ be the matrix whose rows are the
gradients of all equality constraints, followed by the gradients of the
inequality constraints indexed by $I$:
\begin{equation}
    A_I(P,x)
    =
    \begin{pmatrix}
        \nabla h_{P,1}(x)^\top\\
        \vdots\\
        \nabla h_{P,p(}x)^\top\\
        \nabla g_{P,i_1}(x)^\top\\
        \vdots\\
        \nabla g_{P,i_{|I|}}(x)^\top\\
    \end{pmatrix}.
\end{equation}

The order of the indices in $I$ is arbitrary but fixed; the quantities below
are independent of this order.

\item Define
\begin{equation}
    M_I(P,x)
    =
    A_I(P,x)-\bar s_I(P,x)v_x^\top.
\end{equation}
The rank-defect cost associated with $I$ is
\begin{equation}
\Delta_I(P,x)
=
\begin{cases}
\sigma_{\min}\!\left(M_I(P,x)C_x^{-1/2}\right),
& \text{if } p+|I|\le n,\\[0.5em]
0,
& \text{if } p+|I|>n.
\end{cases}
\end{equation}
\item Define
\begin{equation}
    \Theta_I(P,x)^2
    =
    \mu_x\|s_I(P,x)\|^2
    +
    \Delta_I(P,x)^2.
\end{equation}
\end{enumerate}
\end{dfn}
\begin{dfn}
For a fixed labelled representative of $P$ and an index set
$I\subseteq\{1,\ldots,q\}$, let $\mathcal B_I^J(P,x)$ denote the set of
jet formulations $Z\in\mathcal C_{p,q}^J(x)$ for which there exists a
labelled representative
\begin{equation}
    H_Z=((a_1,\xi_1),\dots,(a_p,\xi_p)),
    \qquad
    G_Z=((c_1,\zeta_1),\dots,(c_q,\zeta_q)),
\end{equation}
matched with the chosen labelled representative of $P$, such that
\begin{equation}
    a_\alpha=0,
    \qquad
    \alpha=1,\ldots,p,
\end{equation}
\begin{equation}
    c_i\le0,
    \qquad
    i=1,\ldots,q,
\end{equation}
\begin{equation}
    c_i=0,
    \qquad
    i\in I,
\end{equation}
and the vectors
\begin{equation}
    \xi_1,\ldots,\xi_p,\ \zeta_i\ (i\in I)
\end{equation}
are linearly dependent. Equivalently, $Z$ is feasible at $x$, the
inequalities indexed by $I$ are active at $x$, and LICQ fails for the
subsystem consisting of all equality constraints and the inequalities in $I$.
\end{dfn}    
Although these quantities are defined using labelled representatives, they are
invariant under simultaneous permutations of the inequality labels. Hence the
minimum over all index sets $I\subseteq\{1,\ldots,q\}$ is well-defined on the
quotient formulation space.
\begin{prop}
\label{prop:theta-candidate-cost}
Let $P=(H_P,G_P)\in\mathcal C_{p,q}$ and $x\in K(P)$, and set
$Z=\widehat\Pi_x^1(P)$.

The following assertions hold.

\begin{enumerate}
\item For every index set $I\subseteq\{1,\ldots,q\}$ with $p+|I|\ge1$,
there exists $Z_I^*\in\mathcal B_I^J(P,x)\subset\mathcal B_{p,q}^J(x)$
such that
\begin{equation}
    \widetilde d_{p,q}(Z,Z_I^*)
    \le
    \Theta_I(P,x).
\end{equation}

\item For every $Y\in\mathcal B_{p,q}^J(x)$, there exists an index set
$I\subseteq\{1,\ldots,q\}$ with $p+|I|\ge1$ such that
\begin{equation}
    \widetilde d_{p,q}(Z,Y)
    \ge
    \Theta_I(P,x).
\end{equation}
\end{enumerate}
\end{prop}
\begin{proof}
We first prove assertion (1). Fix an index set
$I\subseteq\{1,\ldots,q\}$ with $p+|I|\ge1$. Write
\begin{equation}
    I=\{i_1,\ldots,i_k\},
    \qquad k=|I|,
\end{equation}
in an arbitrary but fixed order, and set $m=p+k$. Set
\begin{equation}
    B=M_I(P,x)C_x^{-1/2}.
\end{equation}

If $m\le n$, choose a matrix $F^*\in\mathbb R^{m\times n}$ satisfying
\begin{equation}
    \operatorname{rank}(B+F^*)<m,
    \qquad
    \|F^*\|_F=\sigma_{\min}(B).
\end{equation}
Such an $F^*$ exists by the Eckart--Young--Mirsky theorem. If $m>n$, the
rank-defect condition is automatic, and we take
\begin{equation}
    F^*=0.
\end{equation}
In both cases,
\begin{equation}
    \|F^*\|_F=\Delta_I(P,x),
    \qquad
    \operatorname{rank}(B+F^*)<m.
\end{equation}

Define
\begin{equation}
    E^*=F^* C_x^{1/2}
\end{equation}
and
\begin{equation}
    \Xi^*=E^*-\bar s_I(P,x)v_x^\top.
\end{equation}
Then
\begin{equation}
    A_I(P,x)+\Xi^*
    =
    A_I(P,x)-\bar s_I(P,x)v_x^\top+E^*
    =
    M_I(P,x)+E^*.
\end{equation}
Multiplying by $C_x^{-1/2}$, we obtain
\begin{equation}
    (A_I(P,x)+\Xi^*)C_x^{-1/2}
    =
    M_I(P,x)C_x^{-1/2}+E^*C_x^{-1/2}
    =
    B+F^*.
\end{equation}
Hence
\begin{equation}
    \operatorname{rank}(A_I(P,x)+\Xi^*)<m.
\end{equation}

We now define a jet formulation $Z_I^*$. For the equality functions, set
\begin{equation}
    z_\alpha^H
    =
    \bigl(0,\nabla h_\alpha(x)+\Xi_\alpha^*\bigr),
    \qquad \alpha=1,\ldots,p,
\end{equation}
where $\Xi_\alpha^*$ denotes the $\alpha$-th row of $\Xi^*$. For the
inequality functions indexed by $I$, set
\begin{equation}
    z_{i_\ell}^G
    =
    \bigl(0,\nabla g_{i_\ell}(x)+\Xi_{p+\ell}^*\bigr),
    \qquad \ell=1,\ldots,k.
\end{equation}
All inequality functions not indexed by $I$ are left unchanged:
\begin{equation}
    z_j^G
    =
    \bigl(g_j(x),\nabla g_j(x)\bigr),
    \qquad j\notin I.
\end{equation}

Then $Z_I^*$ is feasible at $x$, the inequality functions indexed by $I$
are active at $x$, and the gradients of all equality functions together with
the gradients of the inequality functions indexed by $I$ are linearly
dependent. Thus
\begin{equation}
    Z_I^*\in\mathcal B_I^J(P,x)\subset\mathcal B_{p,q}^J(x).
\end{equation}

The squared cost along the prescribed labelled matching is
\begin{equation}
\begin{aligned}
    &\mu_x\|s_I(P,x)\|^2
    +
    \|(\Xi^*+\bar s_I(P,x)v_x^\top)C_x^{-1/2}\|_F^2  \\
    &\qquad =
    \mu_x\|s_I(P,x)\|^2
    +
    \|E^*C_x^{-1/2}\|_F^2  \\
    &\qquad =
    \mu_x\|s_I(P,x)\|^2
    +
    \|F^*\|_F^2  \\
    &\qquad =
    \mu_x\|s_I(P,x)\|^2
    +
    \Delta_I(P,x)^2  \\
    &\qquad =
    \Theta_I(P,x)^2.
\end{aligned}
\end{equation}
Since $\widetilde d_{p,q}$ is obtained by minimizing over all admissible
matchings, it is no larger than the cost of this prescribed matching. Hence
\begin{equation}
    \widetilde d_{p,q}(Z,Z_I^*)
    \le
    \Theta_I(P,x).
\end{equation}
This proves assertion (1).

We next prove assertion (2). Let $Y\in\mathcal B_{p,q}^J(x)$ be arbitrary.
Choose labelled representatives of $Z$ and $Y$ realizing the optimal
matching distance $\widetilde d_{p,q}(Z,Y)$. Since $Y$ fails LICQ at
$x$, there exists an index set $I\subseteq\{1,\ldots,q\}$ with
$p+|I|\ge1$ such that, in this chosen matching, the inequality functions
indexed by $I$ are active at $x$, and the gradients of all equality
functions together with those inequality functions are linearly dependent.

Write
\begin{equation}
    I=\{i_1,\ldots,i_k\},
    \qquad k=|I|,
\end{equation}
and set $m=p+k$. Let $\Xi\in\mathbb R^{m\times n}$ be the gradient
perturbation matrix corresponding to the equality functions and the inequality
functions indexed by $I$, under the chosen optimal matching. Since the
inequality functions indexed by $I$ are active in $Y$, the value
perturbation on these $m$ functions is $\bar s_I(P,x)$. Moreover,
\begin{equation}
    \operatorname{rank}(A_I(P,x)+\Xi)<m.
\end{equation}

Set
\begin{equation}
    E=\Xi+\bar s_I(P,x)v_x^\top.
\end{equation}
Then
\begin{equation}
    A_I(P,x)+\Xi
    =
    M_I(P,x)+E.
\end{equation}
Therefore
\begin{equation}
    \operatorname{rank}(M_I(P,x)+E)<m.
\end{equation}
Equivalently,
\begin{equation}
    \operatorname{rank}
    \bigl(M_I(P,x)C_x^{-1/2}+EC_x^{-1/2}\bigr)
    <m.
\end{equation}
By the definition of $\Delta_I(P,x)$, or equivalently by the
Eckart--Young--Mirsky theorem,
\begin{equation}
    \|EC_x^{-1/2}\|_F
    \ge
    \Delta_I(P,x).
\end{equation}

Since $\widetilde d_{p,q}(Z,Y)$ is computed using the chosen optimal matching,
and since the remaining constraints contribute nonnegative cost, we obtain
\begin{equation}
\begin{aligned}
    \widetilde d_{p,q}(Z,Y)^2
    &\ge
    \mu_x\|s_I(P,x)\|^2
    +
    \|EC_x^{-1/2}\|_F^2  \\
    &\ge
    \mu_x\|s_I(P,x)\|^2
    +
    \Delta_I(P,x)^2 \\
    &=
    \Theta_I(P,x)^2.
\end{aligned}
\end{equation}
Hence
\begin{equation}
    \widetilde d_{p,q}(Z,Y)
    \ge
    \Theta_I(P,x).
\end{equation}
This proves assertion (2).
\end{proof}
The quantity $\Theta_I(P,x)$ measures the first-order cost of producing an
LICQ failure involving all equality constraints and the inequalities indexed by
$I$. Its value component
$\mu_x\|s_I(P,x)\|^2$ is the cost of making those inequalities active at
$x$, while its gradient component $\Delta_I(P,x)^2$ is the squared distance
from the corrected active-gradient matrix to the set of rank-deficient matrices.
Thus the formula below states that the LICQ radius is the minimum cost, over
all possible active subsystems, of creating a first-order LICQ failure.
\begin{thm}[Explicit LICQ radius formula]
\label{thm:explicit-radius}
Let \(P=(H_P,G_P)\in\mathcal C_{p,q}\), and let \(x\in K(P)\).
Choose labelled representatives
\begin{equation}
    H_P=(h_1,\ldots,h_p),
    \qquad
    G_P=(g_1,\ldots,g_q).
\end{equation}
Then
\begin{equation}
    r_{\mathrm{LICQ},W}^{p,q}(P,x)
    =
    \min_{\substack{I\subseteq\{1,\ldots,q\}\\ p+|I|\ge 1}}
    \Theta_I(P,x).
\end{equation}
If the index set of the minimum is empty, the minimum is understood to be
$+\infty$.
\end{thm}
\begin{proof}
Set $Z=\widehat\Pi_x^1(P)$. By
Proposition~\ref{prop:equivalence-radii}, it is enough to show that
\begin{equation}
    r_{\mathrm{LICQ},J}^{p,q}(Z,x)
    =
    \min_{\substack{I\subseteq\{1,\ldots,q\}\\ p+|I|\ge1}}
    \Theta_I(P,x).
\end{equation}

First, fix an index set $I\subseteq\{1,\ldots,q\}$ with $p+|I|\ge1$.
By Proposition~\ref{prop:theta-candidate-cost}(1), there exists
$Z_I^*\in\mathcal B_{p,q}^J(x)$ such that
\begin{equation}
    \widetilde d_{p,q}(Z,Z_I^*)
    \le
    \Theta_I(P,x).
\end{equation}
Therefore
\begin{equation}
    r_{\mathrm{LICQ},J}^{p,q}(Z,x)
    \le
    \Theta_I(P,x).
\end{equation}
Taking the minimum over all such $I$, we obtain
\begin{equation}
    r_{\mathrm{LICQ},J}^{p,q}(Z,x)
    \le
    \min_{\substack{I\subseteq\{1,\ldots,q\}\\ p+|I|\ge1}}
    \Theta_I(P,x).
\end{equation}

Conversely, let $Y\in\mathcal B_{p,q}^J(x)$ be arbitrary. By
Proposition~\ref{prop:theta-candidate-cost}(2), there exists an index set
$I\subseteq\{1,\ldots,q\}$ with $p+|I|\ge1$ such that
\begin{equation}
    \widetilde d_{p,q}(Z,Y)
    \ge
    \Theta_I(P,x).
\end{equation}
Hence
\begin{equation}
    \widetilde d_{p,q}(Z,Y)
    \ge
    \min_{\substack{I'\subseteq\{1,\ldots,q\}\\ p+|I'|\ge1}}
    \Theta_{I'}(P,x).
\end{equation}
Taking the infimum over $Y\in\mathcal B_{p,q}^J(x)$, we obtain
\begin{equation}
    r_{\mathrm{LICQ},J}^{p,q}(Z,x)
    \ge
    \min_{\substack{I\subseteq\{1,\ldots,q\}\\ p+|I|\ge1}}
    \Theta_I(P,x).
\end{equation}
Combining the two inequalities proves the formula.
\end{proof}
\section{Augmentation formulas for LICQ radii}
\label{sec:augmentation}

For $\phi\in W$, we write $P^{\phi\le}$ for the formulation obtained from
$P$ by adding the inequality constraint $\phi\le 0$, and $P^{\phi=}$ for
the formulation obtained by adding the equality constraint $\phi=0$. The
addition is understood in the sense of unordered collections.

The augmentation formulas are direct consequences of the explicit radius
formula. When an inequality constraint is added, a nearest LICQ failure either
ignores the new inequality or uses it as part of the active subsystem. The first
case gives the original radius, while the second case gives the additional
terms involving $\Theta_{I\cup\{q+1\}}$. When an equality constraint is added,
the new equality is automatically included in every active-gradient subsystem,
so the radius is obtained by applying the explicit formula to the augmented
formulation.
\begin{thm}[Addition of an inequality constraint]
\label{thm:augmentation-ineq}
Let $P=(H_P,G_P)\in\mathcal C_{p,q}$ and let $x\in K(P)$. Let
$\phi\in W$ satisfy $\phi(x)\le 0$. Then we have
\begin{equation}
    r_{\mathrm{LICQ},W}^{p,q+1}(P^{\phi\le},x)
    =
    \min\left\{
        r_{\mathrm{LICQ},W}^{p,q}(P,x),
        \;
        \min_{I\subseteq\{1,\ldots,q\}}
        \Theta_{I\cup\{q+1\}}(P^{\phi\le},x)
    \right\}.
\end{equation}

\end{thm}

\begin{proof}
Choose labelled representatives
\begin{equation}
    G_P=(g_1,\ldots,g_q),
    \qquad
    G_{P^{\phi\le}}=(g_1,\ldots,g_q,\phi).
\end{equation}
Apply Theorem~\ref{thm:explicit-radius} to the augmented formulation
$P^{\phi\le}$. We obtain
\begin{equation}
    r_{\mathrm{LICQ},W}^{p,q+1}(P^{\phi\le},x)
    =
    \min_{\substack{L\subseteq\{1,\ldots,q+1\}\\ p+|L|\ge1}}
    \Theta_L(P^{\phi\le},x).
\end{equation}
Every index set $L\subseteq\{1,\ldots,q+1\}$ falls into exactly one of the
following two classes.

First, suppose $q+1\notin L$. Then $L\subseteq\{1,\ldots,q\}$, and the
corresponding cost is the same as in the original formulation:
\begin{equation}
    \Theta_L(P^{\phi\le},x)=\Theta_L(P,x).
\end{equation}
Therefore the minimum over all such $L$ is
\begin{equation}
    r_{\mathrm{LICQ},W}^{p,q}(P,x)
\end{equation}
by Theorem~\ref{thm:explicit-radius}.

Second, suppose $q+1\in L$. Then $L$ can be written uniquely as
\begin{equation}
    L=I\cup\{q+1\}
\end{equation}
for some $I\subseteq\{1,\ldots,q\}$. Hence the minimum over all such $L$ is
\begin{equation}
    \min_{I\subseteq\{1,\ldots,q\}}
    \Theta_{I\cup\{q+1\}}(P^{\phi\le},x).
\end{equation}
Taking the minimum over the two classes gives the desired formula.
\end{proof}

\begin{thm}[Addition of an equality constraint]
\label{thm:augmentation-eq}
Let $P=(H_P,G_P)\in\mathcal C_{p,q}$ and let $x\in K(P)$. Let
$\phi\in W$ satisfy $\phi(x)=0$. Then $x\in K(P^{\phi=})$, and
\begin{equation}
    r_{\mathrm{LICQ},W}^{p+1,q}(P^{\phi=},x)
    =
    \min_{I\subseteq\{1,\ldots,q\}}
    \Theta_I(P^{\phi=},x).
\end{equation}
\end{thm}

\begin{proof}
Apply Theorem~\ref{thm:explicit-radius} to $P^{\phi=}$. The equality
collection of $P^{\phi=}$ is obtained by adding the equality function $\phi$,
while the inequality collection remains indexed by
$\{1,\ldots,q\}$. Therefore
\begin{equation}
    r_{\mathrm{LICQ},W}^{p+1,q}(P^{\phi=},x)
    =
    \min_{\substack{I\subseteq\{1,\ldots,q\}\\ (p+1)+|I|\ge1}}
    \Theta_I(P^{\phi=},x).
\end{equation}
Since $(p+1)+|I|\ge1$ holds for every
$I\subseteq\{1,\ldots,q\}$, this reduces to
\begin{equation}
    r_{\mathrm{LICQ},W}^{p+1,q}(P^{\phi=},x)
    =
    \min_{I\subseteq\{1,\ldots,q\}}
    \Theta_I(P^{\phi=},x).
\end{equation}
The added equality function $\phi$ is automatically included in the equality
part of every active-gradient matrix appearing in $\Theta_I(P^{\phi=},x)$.
This proves the formula.
\end{proof}

\end{document}